\documentclass[11pt,reqno]{amsart}
\usepackage[a4paper,margin=29mm]{geometry}
\usepackage[T1]{fontenc}
\usepackage{lmodern}
\usepackage{amsmath,amssymb,mathtools}
\usepackage{microtype}
\usepackage[unicode,colorlinks=true,linkcolor=blue!45!black,citecolor=blue!45!black,urlcolor=blue!45!black]{hyperref}
\usepackage{xcolor}
\numberwithin{equation}{section}
\newtheorem{theorem}{Theorem}[section]
\newtheorem{proposition}[theorem]{Proposition}
\newtheorem{lemma}[theorem]{Lemma}
\newtheorem{corollary}[theorem]{Corollary}
\theoremstyle{remark}

\newcommand{\R}{\mathbb R}
\newcommand{\Z}{\mathbb Z}
\newcommand{\T}{\mathbb T}
\newcommand{\HH}{\mathbb H^2}
\newcommand{\dd}{\,\mathrm d}
\newcommand{\eps}{\varepsilon}
\newcommand{\supp}{\operatorname{supp}}
\newcommand{\Sch}{\mathcal S}
\newcommand{\norm}[2][]{\left\lVert #2\right\rVert_{#1}}

\newcommand{\br}[1]{\left\langle #1\right\rangle}
\newcommand{\U}{U}
\newcommand{\DHyper}{D_H}
\newcommand{\X}{\mathcal X}

\hypersetup{pdftitle={Schrodinger estimates with arbitrarily small periodic regularity on H2 times T},pdfauthor={Xuyi Wei},pdfsubject={Anisotropic Strichartz estimates},pdfkeywords={Schrodinger equation, hyperbolic space, periodic direction, decoupling}}
\title[Strichartz estimates on $\HH\times\T$]{Schr\"odinger estimates with arbitrarily small periodic regularity on $\HH\times\T$}
\author{Xuyi Wei}
\address{School of Mathematical Sciences, Beihang University, Beijing, China}
\email{weixuyi@buaa.edu.cn}

\date{}
\begin{document}
\begin{abstract}
Let $M=\mathbb H^2\times\mathbb T$ carry its product metric. We prove that the Schr\"odinger evolution maps $L^2(\mathbb H^2;H^s(\mathbb T))$ into $L^2([0,1];L^6(M))$ for every $s>0$, without any regularity assumption in the hyperbolic variable. The proof separates hyperbolic and periodic frequencies. A localized hyperbolic kernel estimate restricts the effective interaction time at hyperbolic frequency $L$ to order $L^{-1}$. Quadratic exponential sums then give a logarithmic bound when the periodic bandwidth is comparable to $L$. A mixed-norm decoupling argument, including an $L^6$-valued extension and a periodic transference step, reduces wider periodic bands to this case. The infimum of the admissible periodic regularities is therefore zero. The estimate at zero regularity is not established here.
\end{abstract}
\maketitle

\section{Introduction}

Set $\T=\R/(2\pi\Z)$ and equip
\[
 M=\HH\times\T
\]
with the product metric. The hyperbolic metric has constant curvature $-1$. We use the nonpositive Laplace--Beltrami operator, so that
\begin{equation}\label{eq:product}
 \Delta_M=\Delta_H+\partial_\theta^2,
 \qquad
 \U(t)=e^{it\Delta_M}=e^{it\Delta_H}e^{it\partial_\theta^2}.
\end{equation}
All $L^p$ spaces on $M$ are defined using the product Riemannian measure. Write
\[
 \br{D_\theta}^{s}=(1-\partial_\theta^2)^{s/2}.
\]
The pair $(2,6)$ is the three-dimensional Schr\"odinger endpoint. The product geometry nevertheless contains only two noncompact directions, while the periodic evolution has rational-time concentration. The purpose of this paper is to obtain the endpoint space-time exponents with an arbitrarily small regularity cost confined to the periodic variable.

\begin{theorem}\label{thm:main}
For every $s>0$ there is a constant $C_s$ such that
\begin{equation}\label{eq:main}
 \norm[{L^2([0,1];L^6(M))}]{\U(t)f}
 \le C_s\norm[{L^2(M)}]{\br{D_\theta}^{s}f}.
\end{equation}
Consequently, if
\[
 \mathcal A=\{s\in\R:\text{\eqref{eq:main} holds with a constant independent of }f\},
\]
then $(0,\infty)\subset\mathcal A\subset[0,\infty)$ and $\inf\mathcal A=0$.
No assertion is made about membership of $0$ in $\mathcal A$.
\end{theorem}

The estimate extends to any bounded time interval. In particular, for $T>0$,
\begin{equation}\label{eq:finite-time}
 \norm[{L^2([0,T];L^6(M))}]{\U(t)f}
 \le C_s\sqrt{\lceil T\rceil}\,
 \norm[{2}]{\br{D_\theta}^{s}f}.
\end{equation}
Indeed, one applies \eqref{eq:main} on successive intervals of length at most one, using unitarity and commutation with $\br{D_\theta}^{s}$.

\subsection{Relation to previous work}

The Euclidean estimate at the pair $(2,6)$ is part of the endpoint Strichartz theory of Keel and Tao \cite{KT}. The usual abstract implication from a three-dimensional dispersive bound cannot be applied directly to \eqref{eq:product}: the unlocalized periodic propagator has rational-time concentration. Periodic Fourier restriction estimates originating in Bourgain's work \cite{Bourgain93} and the derivative-loss estimates on compact manifolds of Burq, G\'erard, and Tzvetkov \cite{BGT} provide two relevant points of comparison. The latter results concern fully compact geometries and do not by themselves determine an anisotropic loss on the noncompact product considered here.

A compact factor does not force a positive regularity threshold for every space-time norm. Takaoka and Tzvetkov \cite{TT} proved a local lossless $L^4$ estimate on $\R\times\T$. Barron, Christ, and Pausader \cite{BCP} obtained a corresponding global estimate with an $\ell^8$ summation over unit time intervals. Barron \cite{Barron} studied global estimates on more general Euclidean--periodic products using norms on unit time strips. These results concern other space-time norms and geometries; in particular, they do not settle the mixed endpoint $L_t^2L_{x,\theta}^6$ on $\HH\times\T$.

Dispersive estimates on products can be obtained by combining suitable estimates on the factors; see Pierfelice \cite{Pierfelice}. In the present proof this idea is used after frequency localization, when the periodic kernel becomes a finite smooth quadratic sum. We do not assume a dispersive estimate for the unlocalized circle evolution. The tensor factorization is used only in the localized $TT^*$ kernel, after which the arithmetic time integral is estimated separately.

On the hyperbolic side, the spherical transform and the Harish--Chandra expansion are classical; see Helgason \cite{Helgason}. Anker and Pierfelice \cite{AP} established dispersive and Strichartz estimates on real hyperbolic spaces. Chen and Hassell \cite{CH} developed oscillatory spectral-measure estimates on nontrapping asymptotically hyperbolic manifolds, and Germain and L\'eger \cite{GL} studied spectral projectors, resolvents, and Fourier restriction on hyperbolic space. We use the explicit real-hyperbolic structure to prove a frequency-dependent kernel bound in the precise form needed here. The differentiated spherical amplitudes are established in Appendix~\ref{app:spherical}, rather than inferred from an undifferentiated asymptotic formula.

There is an elementary bound with a one-third derivative in the periodic variable, obtained from $H^{1/3}(\T)\hookrightarrow L^6(\T)$ and a vector-valued hyperbolic Strichartz estimate; its proof is included in Section~\ref{sec:baseline}. This argument discards the periodic time oscillation. The purpose of the arithmetic and decoupling steps below is to retain that oscillation and reduce the regularity cost to every positive $s$.

The harmonic-analytic input is the parabola $\ell^2$ decoupling theorem of Bourgain and Demeter \cite{BD}. Its mixed-norm form was established by Dasu, Jung, Li, and Madrid \cite{DJLM}. We give the particular extension needed here: the additional hyperbolic variable remains inside the spatial $L^6$ norm. Applying scalar mixed-norm decoupling pointwise in that variable and then interchanging norms would not justify this step. Instead, we start from pure $L^6$ decoupling and prove the reduction of the time exponent directly by a reproducing-kernel argument.

\subsection{Outline and analytic inputs}

The proof has three components. At hyperbolic frequency $L$, the localized kernel obeys
\begin{equation}\label{eq:intro-kernel}
 \norm[{L^1(\HH)\to L^\infty(\HH)}]{e^{it\Delta_H}\psi(\DHyper/L)}
 \lesssim_K |t|^{-1}(1+L|t|)^{-K},
 \qquad \DHyper=\sqrt{-\Delta_H-\tfrac14}.
\end{equation}
Together with quadratic exponential-sum estimates, this yields a bound of order $\sqrt{\log(2L)}$ when the periodic bandwidth is $O(L)$. If the periodic frequency scale is $N\ge L$, decoupling divides it into intervals of width $L$ at cost $(N/L)^\eps$. The remaining frequency sums converge after an arbitrarily small positive power of $N$ is allowed.

Spectral inversion and the scalar parabola decoupling theorem are used as established inputs. The kernel estimates, arithmetic summation, vector-valued extension, periodic transference, and frequency reconstruction used in the argument are proved below. Appendix~\ref{app:spherical} supplies the uniform spherical-function symbol bounds needed for repeated integration by parts.

\section{Spectral localization and the hyperbolic kernel}

\subsection{Fourier conventions and the anisotropic data space}

The coordinate on $\T=\R/(2\pi\Z)$ is integrated against $\dd\theta$, without normalization. Thus
\begin{equation}\label{eq:fourier-series-convention}
 f_n(x)=\frac1{2\pi}\int_0^{2\pi}f(x,\theta)e^{-in\theta}\dd\theta,
 \qquad f(x,\theta)=\sum_{n\in\Z}f_n(x)e^{in\theta},
\end{equation}
and Parseval's identity gives
\begin{equation}\label{eq:data-norm}
 \norm[{2}]{\br{D_\theta}^s f}^2
 =2\pi\sum_{n\in\Z}(1+n^2)^s\norm[{L^2(\HH)}]{f_n}^2.
\end{equation}
This identity defines the data space by completion and makes explicit that no hyperbolic derivative is present. On Euclidean variables our Fourier convention is
\[
 \widehat F(\tau,\xi)=\int_{\R^2}e^{-i(t\tau+y\xi)}F(t,y)\dd t\dd y,
 \qquad
 F(t,y)=\frac1{(2\pi)^2}\int_{\R^2}e^{i(t\tau+y\xi)}\widehat F(\tau,\xi)\dd\tau\dd\xi.
\]
We reserve $\lambda\ge0$ for the spectral parameter of $\DHyper=\sqrt{-\Delta_H-1/4}$. The self-adjoint realization of $-\Delta_H$ on the complete hyperbolic plane has spectrum $[1/4,\infty)$; its spectral resolution defines $\DHyper$ and all the multipliers used below. The functional calculi in $x$ and $\theta$ commute, since they act on distinct factors. In particular,
\begin{equation}\label{eq:mode-evolution}
 \U(t)f=\sum_{n\in\Z}e^{-itn^2}\bigl(e^{it\Delta_H}f_n\bigr)e^{in\theta}
\end{equation}
in $L^2(M)$, and the evolution is unitary on each data space in \eqref{eq:data-norm}.

Unless explicitly stated otherwise, implicit constants are independent of the dyadic parameters, interval centers, and Fourier mode centers. Constants with a subscript $\eps$ may depend on a fixed positive $\eps$. Every sum in a dyadic scale is over powers of two, not over all positive integers.

\subsection{Frequency decompositions}

Let $\mathcal D=\{1,2,4,8,\ldots\}$. Choose a smooth partition of unity $\{p_L\}_{L\in\mathcal D}$ on $[0,\infty)$ such that $p_1$ is supported in $[0,2]$ and, for $L\ge2$, $p_L$ is supported in $[L/2,2L]$. The dilated symbols have bounded seminorms, and the supports overlap a bounded number of times. Put $P_L=p_L(\DHyper)$. The spectral theorem gives
\begin{equation}\label{eq:H-L2}
 f=\sum_{L\in\mathcal D}P_Lf\quad\text{in }L^2,
 \qquad
 \sum_L\norm[{2}]{P_Lf}^2\lesssim\norm[{2}]{f}^2.
\end{equation}
Choose an analogous smooth Fourier decomposition $Q_N=q_N(D_\theta)$ on $\T$, with $N=1$ containing the low modes. For example, take a smooth even function $q_1$ equal to one on $[-1,1]$ and supported in $[-2,2]$, and put
\[
 q_{2^j}(\xi)=q_1(2^{-j}\xi)-q_1(2^{-j+1}\xi),\qquad j\ge1.
\]
The corresponding partial sums telescope to the identity. The support of $q_N$ is contained in $\{N/2\le|\xi|\le2N\}$ for $N\ge2$, and only a bounded number of symbols are nonzero at any one frequency. Enlarged cutoffs, equal to one on the indicated supports, will be used when forming $TT^*$.

For later use, Parseval gives the quantitative bound
\begin{align}\label{eq:theta-weighted-parseval}
 \sum_{N\in\mathcal D}N^{2s}\norm[{2}]{Q_Nf}^2
 &=2\pi\sum_{n\in\Z}\Bigl(\sum_{N\in\mathcal D}N^{2s}|q_N(n)|^2\Bigr)
                   \norm[{L^2(\HH)}]{f_n}^2\notag\\
 &\le C_s\norm[{2}]{\br{D_\theta}^sf}^2.
\end{align}
Indeed, on each nonzero summand $N\asymp\br n$, and the number of summands is uniformly bounded. The same reasoning in the spectral measure of $\DHyper$ proves \eqref{eq:H-L2}; no assertion about $L^p(\HH)$ multipliers enters this argument.

No $L^p$ Littlewood--Paley theorem for compactly supported hyperbolic spectral multipliers is used. In the hyperbolic variable, reconstruction will rely only on \eqref{eq:H-L2} and Cauchy--Schwarz. In the periodic variable, the usual Littlewood--Paley theorem is available.

For a time interval $I$, write
\begin{equation}\label{eq:X-def}
 \X(I)=L^2(I;L^6(M)),\qquad
 \X'(I)=L^2(I;L^{6/5}(M)).
\end{equation}
For the Euclidean time Fourier transform we use the phase $e^{-it\tau}$. Thus a factor $e^{-it(n^2+\lambda^2+1/4)}$ has time frequency $\tau=-n^2-\lambda^2-1/4$.

\subsection{Spherical-function bounds}

Let $r$ denote hyperbolic distance. The spherical inversion formula \cite[Chapter IV]{Helgason} gives the radial kernel of $m(\DHyper)$ as
\begin{equation}\label{eq:spherical-inversion}
 k_m(r)=c_H\int_0^\infty m(\lambda)\varphi_\lambda(r)
                \lambda\tanh(\pi\lambda)\dd\lambda,
\end{equation}
where $c_H>0$ depends only on the normalization of measure. We use the following standard spherical-function estimates; they also follow from the spectral-measure structure in \cite[Theorem 1.3]{CH}, specialized to real hyperbolic space.

\begin{lemma}[Spherical-function bounds]\label{lem:spherical}
For $\lambda\ge1$, the following assertions hold.
\begin{enumerate}
\item The spherical function has the boundary integral representation
\begin{equation}\label{eq:boundary-spherical}
 \varphi_\lambda(r)=\frac1{2\pi}\int_0^{2\pi}
 (\cosh r-\sinh r\cos\alpha)^{-1/2-i\lambda}\dd\alpha.
\end{equation}
In particular, $|\varphi_\lambda(r)|\le\varphi_0(r)\le1$.
\item In the region $\lambda r\ge1$, one may write
\begin{equation}\label{eq:spherical-symbol}
 \lambda\tanh(\pi\lambda)\varphi_\lambda(r)
 =\sum_{\pm}e^{\pm i\lambda r}b_\pm(\lambda,r),
\end{equation}
where
\begin{equation}\label{eq:b-symbol}
 |\partial_\lambda^j b_\pm(\lambda,r)|
 \le C_j\lambda^{1/2-j}(\sinh r)^{-1/2},\qquad j\ge0.
\end{equation}
The estimates hold on slightly enlarged regions $\lambda r\ge c>0$, with constants depending on $c$.
\end{enumerate}
\end{lemma}

The bound by one in the first part can be checked directly. Put $P_r(\alpha)=(\cosh r-\sinh r\cos\alpha)^{-1}$. Its average is one, since
\[
 \frac1{2\pi}\int_0^{2\pi}\frac{\dd\alpha}{\cosh r-\sinh r\cos\alpha}
 =\bigl(\cosh^2r-\sinh^2r\bigr)^{-1/2}=1.
\]
Taking absolute values in \eqref{eq:boundary-spherical} and using Cauchy--Schwarz gives
\[
 |\varphi_\lambda(r)|\le\frac1{2\pi}\int_0^{2\pi}P_r(\alpha)^{1/2}\dd\alpha\le1.
\]
This argument also covers bounded real $\lambda$, which will be used for the low-frequency cutoff.

The two-phase expansion is uniform both near the diagonal and at large distance. A proof using the Mehler--Dirichlet formula for $r\le1$ and the radial differential equation for $r\ge1$ is given in Appendix~\ref{app:spherical}.

\begin{lemma}[Frequency-localized hyperbolic dispersion]\label{lem:H-kernel}
Let $\psi\in C_c^\infty((a,b))$, where $0<a<b<\infty$. For $L\ge2$, every $K\ge0$, and $0<|t|\le1$,
\begin{equation}\label{eq:H-kernel}
 \norm[{L^1(\HH)\to L^\infty(\HH)}]{e^{it\Delta_H}\psi(\DHyper/L)}
 \le C_{\psi,K}|t|^{-1}(1+L|t|)^{-K}.
\end{equation}
The same assertion, with $L=1$, holds for a fixed smooth cutoff supported in a bounded spectral interval containing zero.
\end{lemma}

\begin{proof}
Complex conjugation reduces the proof to $t>0$. The harmless factor $e^{-it/4}$ will be omitted. By \eqref{eq:spherical-inversion}, the kernel is
\begin{equation}\label{eq:kL}
 k_L(t,r)=c_H\int_0^\infty e^{-it\lambda^2}
 \psi(\lambda/L)\varphi_\lambda(r)\lambda\tanh(\pi\lambda)\dd\lambda.
\end{equation}
The bounded range of $L$ for which the spectral support meets $[0,1]$ is handled directly: spectral inversion gives a uniformly bounded kernel, and $0<t\le1$ while $L$ stays bounded. We may therefore assume $\lambda\ge1$ on the spectral support. Fix $C_0\ge2/a$, choose $0<c<a$, and choose $A_0\ge\max\{1,2/(ac)\}$. The bound $|\varphi_\lambda|\le1$ gives $|k_L|\lesssim L^2$. If $tL^2\le A_0$, then $Lt$ is bounded and this proves \eqref{eq:H-kernel}.

Assume henceforth $tL^2>A_0$. In the region $r<ctL$, both phases $-t\lambda^2\pm\lambda r$ are nonstationary on the spectral support, with first derivative bounded below by a constant multiple of $tL$.

First suppose $r\le C_0/L$. In \eqref{eq:boundary-spherical}, put
\[
 B(r,\alpha)=\log(\cosh r-\sinh r\cos\alpha).
\]
Then $|B|\le r$, and the phase in the $\lambda$ integral is $-t\lambda^2-\lambda B$. After the substitution $\lambda=L\mu$, its first derivative has size at least $ctL^2$. All derivatives of the cutoff amplitude are uniformly controlled after its overall factor $L^2$ is removed, and $e^{-B/2}$ is bounded in this region. Repeated integration by parts therefore yields
\begin{equation}\label{eq:H-nonstat-small}
 |k_L(t,r)|\le C_M L^2(tL^2)^{-M}.
\end{equation}
If instead $r>C_0/L$, insert \eqref{eq:spherical-symbol}. The same rescaling and integration by parts give
\begin{equation}\label{eq:H-nonstat-large}
 |k_L(t,r)|\le C_M L^{3/2}(\sinh r)^{-1/2}(tL^2)^{-M}
 \lesssim L^2(tL^2)^{-M}.
\end{equation}
To make the integration-by-parts constants explicit, put
\[
 A=tL^2,\qquad \Phi_\pm(\mu)=-A\mu^2\pm Lr\mu.
\]
On the support of $\psi$, the inequalities $r<ctL$ and $0<c<a$ give
\[
 |\Phi_\pm'(\mu)|\ge(2a-c)A,
 \qquad |\Phi_\pm''(\mu)|=2A.
\]
Every derivative of $(\Phi_\pm')^{-1}$ is consequently bounded by
$C_jA^{-1}$.  If $a_L$ is supported in the same fixed compact interval,
the formal transpose of
\[
 \mathcal L_\pm=\frac1{i\Phi_\pm'(\mu)}\partial_\mu
\]
satisfies
\[
 \| (\mathcal L_\pm^*)^M a_L\|_{L^1_\mu}
 \le C_M A^{-M}\sum_{j=0}^M\|\partial_\mu^j a_L\|_{L^1_\mu}.
\]
There are no boundary terms because $\psi$ is compactly supported in
$(a,b)$.  For the two-phase representation, its rescaled amplitude obeys
\[
 \sum_{j=0}^M
 \|\partial_\mu^j[L\psi(\mu)b_\pm(L\mu,r)]\|_{L^1_\mu}
 \lesssim_M L^{3/2}(\sinh r)^{-1/2}.
\]
For the boundary integral, the corresponding amplitude after its overall
factor $L^2$ is removed is
$\psi(\mu)\mu\tanh(\pi L\mu)e^{-B(r,\alpha)/2}$.
Its $\mu$ derivatives are bounded uniformly when $r\le C_0/L$.
The same calculation with
$\Phi(\mu)=-A\mu^2-LB(r,\alpha)\mu$ applies, since
$|B(r,\alpha)|\le r<ctL$.  This proves both nonstationary estimates with
constants independent of $L,t,r,\alpha$.

By choosing $M\ge K+1$, these estimates are bounded by the right side of \eqref{eq:H-kernel}. Indeed, when $Lt\ge1$, one uses $Lt\le tL^2$; when $Lt\le1$, the desired additional factor is bounded below by a constant.

It remains to consider $r\ge ctL$. Our assumption on $tL^2$ guarantees $\lambda r\gtrsim1$ on the spectral support. The second derivative of $-t\lambda^2\pm\lambda r$ is $-2t$. For the stationary-region estimate, let
$a_\pm(\lambda)=\psi(\lambda/L)b_\pm(\lambda,r)$.
The symbol bounds give
\[
 \|a_\pm\|_{L^\infty_\lambda}
 +\|a_\pm'\|_{L^1_\lambda}
 \lesssim L^{1/2}(\sinh r)^{-1/2}.
\]
The assumption $tL^2>A_0$ and $r\ge ctL$ ensures
$\lambda r\ge acA_0\ge2$ throughout the spectral support, so the
two-phase representation is valid on the whole integration interval.

The second-derivative van der Corput estimate (see \cite[Chapter VIII]{Stein}), together with \eqref{eq:b-symbol}, gives
\begin{equation}\label{eq:H-vdc}
 |k_L(t,r)|\lesssim t^{-1/2}L^{1/2}(\sinh r)^{-1/2}
 =t^{-1}(Lt)^{1/2}(\sinh r)^{-1/2}.
\end{equation}
For $Lt\le1$, use $\sinh r\ge r\ge ctL$. For $Lt\ge1$, use $r\ge ctL$ and the exponential growth of $\sinh r$ to obtain
\[
 (Lt)^{1/2}(\sinh r)^{-1/2}\le C_K(1+Lt)^{-K}.
\]
This proves \eqref{eq:H-kernel}. A fixed low-frequency cutoff has a uniformly bounded kernel by \eqref{eq:spherical-inversion}, so the assertion for $L=1$ follows for $0<t\le1$ after changing $C_K$.
\end{proof}

\section{Quadratic sums and time integration}

Quadratic sums enter periodic Schr\"odinger restriction estimates already in \cite{Bourgain93}. We prove the precise smooth bound and weighted time integral required here, keeping track of the numerator count for arcs meeting a short initial interval. The arithmetic estimates are stated in variables for which the time period is one. Write $e(v)=e^{2\pi i v}$. For $a_0\in C_c^\infty(\R)$ define
\begin{equation}\label{eq:W-def}
 W_N(\tau,y)=\sum_{n\in\Z}a_0(n/N)e(ny-n^2\tau),
 \qquad A_N(\tau)=\sup_{y\in\R}|W_N(\tau,y)|.
\end{equation}
All constants below may depend on a fixed finite collection of seminorms of $a_0$ and on its support.

\begin{lemma}[Smooth quadratic sum]\label{lem:weyl}
Suppose $N\ge2$, $(a,q)=1$, $1\le q\le N$, and
\[
 \tau=a/q+\beta,\qquad |\beta|\le(qN)^{-1}.
\]
Then
\begin{equation}\label{eq:weyl}
 A_N(\tau)\le C\frac{N}{\sqrt q\sqrt{1+N^2|\beta|}}.
\end{equation}
\end{lemma}

\begin{proof}
Expand the periodic sequence $e(-an^2/q)$ in a finite Fourier series modulo $q$. Its Fourier coefficients have absolute value at most $\sqrt{2/q}$. To verify this, the corresponding unnormalized Gauss sum is
\[
 G(a,m;q)=\sum_{r\bmod q}e((-ar^2-mr)/q).
\]
On writing $r=s+h$ in $|G|^2$, the sum over $s$ vanishes unless $q\mid2ah$. Since $(a,q)=1$, at most two residue classes $h$ contribute, and $|G|^2\le2q$.

To apply Poisson summation to each term, set
\[
 b=N^2\beta,\qquad I_b(v)=\int_\R a_0(z)e(vz-bz^2)\dd z.
\]
More explicitly, the finite Fourier expansion is
\[
 e(-an^2/q)=\sum_{m=0}^{q-1}c_m e(mn/q),
 \qquad |c_m|\le\sqrt{2/q}.
\]
Poisson summation then gives
\[
 W_N(\tau,y)
 =N\sum_{m=0}^{q-1}c_m\sum_{\ell\in\Z}
 I_b\bigl(N(y+m/q-\ell)\bigr).
\]
As $(m,\ell)$ varies, the integer $k=q\ell-m$ runs through $\Z$ exactly
once.  Taking absolute values gives the asserted single sum over $k$.

Consequently,
\begin{equation}\label{eq:poisson-sum}
 |W_N(\tau,y)|\le C\frac{N}{\sqrt q}
 \sum_{k\in\Z}\left|I_b\bigl(N(y-k/q)\bigr)\right|.
\end{equation}
The van der Corput estimate, with the trivial bound when $|b|\le1$, gives
\begin{equation}\label{eq:Ib-core}
 |I_b(v)|\le C(1+|b|)^{-1/2}.
\end{equation}
For $|v|\ge C_0(1+|b|)$, the phase derivative $v-2bz$ is bounded away from zero on the support of $a_0$. Repeated integration by parts gives, for every integer $M\ge1$,
\begin{equation}\label{eq:Ib-tail}
 |I_b(v)|\le C_M(1+|v|)^{-M}.
\end{equation}
The sample points in \eqref{eq:poisson-sum} have spacing $S=N/q\ge1$, and the major-arc condition implies $|b|\le S$. Consequently, only a bounded number of samples belong to $|v|\le C_0(1+|b|)$. Their contribution is bounded by \eqref{eq:Ib-core}. For the remaining samples, \eqref{eq:Ib-tail} and $S\ge1$ imply
\[
 \sum_{|v_k|>C_0(1+|b|)}(1+|v_k|)^{-M}
 \lesssim(1+|b|)^{1-M}.
\]
Taking $M\ge2$ bounds the full sample sum by $C(1+|b|)^{-1/2}$. Inserting this in \eqref{eq:poisson-sum} proves the lemma.
\end{proof}

\begin{lemma}[Weighted arithmetic integral]\label{lem:arithmetic}
For $N\ge2$ and $0<T\le1/2$,
\begin{equation}\label{eq:J-bound}
 J_N(T):=\int_0^T\tau^{-2/3}A_N(\tau)^{2/3}\dd\tau
 \le C\bigl(\log(2N)+(NT)^{1/3}\bigr).
\end{equation}
Consequently, for $L\ge1$ and any fixed sufficiently large $K$,
\begin{equation}\label{eq:weighted-J}
 \int_0^{1/2}\tau^{-2/3}(1+L\tau)^{-K}A_N(\tau)^{2/3}\dd\tau
 \lesssim \log(2N)+(N/L)^{1/3}.
\end{equation}
\end{lemma}

\begin{proof}
We may assume $N\ge4$, since the remaining values are covered by the trivial bound $A_N\lesssim N$. For integer $N$, the pigeonhole principle applied to the $N+1$ fractional parts $0,\tau,\ldots,N\tau$ gives integers $1\le q\le N$ and $a$ with $|q\tau-a|\le N^{-1}$. Reducing $a/q$ preserves the required inequality because its new denominator is no larger than $q$. Thus Dirichlet approximation covers $[0,T]$ by arcs
\[
 \left|\tau-\frac aq\right|\le\frac1{qN},\qquad
 1\le q\le N,\quad (a,q)=1.
\]
Overlaps cause no difficulty because the integrand is nonnegative and we only seek an upper bound. The zero rational is $a/q=0/1$. On its arc Lemma~\ref{lem:weyl} gives
\[
 A_N(\tau)^{2/3}\lesssim N^{2/3}(1+N^2\tau)^{-1/3}.
\]
The integral over $0<\tau\le N^{-2}$ is $O(1)$. Over $N^{-2}<\tau\le N^{-1}$, the full integrand is $O(\tau^{-1})$. Hence the zero arc contributes $O(\log(2N))$.

On an arc with $a\ge1$, one has $\tau\asymp a/q$, because the relative error is at most $1/(aN)\le1/4$. Moreover, an arc that meets $[0,T]$ satisfies
\[
 a\le qT+1/N,\qquad qT\ge1-1/N,
\]
and therefore $a\le CqT$. By Lemma~\ref{lem:weyl},
\begin{align}\label{eq:arc-integral}
 \int_{|\beta|\le1/(qN)}A_N(a/q+\beta)^{2/3}\dd\beta
 &\lesssim N^{2/3}q^{-1/3}
 \int_0^{1/(qN)}(1+N^2\beta)^{-1/3}\dd\beta\\
 &\lesssim N^{-2/3}q^{-1}.\notag
\end{align}
The last integration can be evaluated explicitly:
\[
 \int_0^{1/(qN)}(1+N^2\beta)^{-1/3}\dd\beta
 =\frac{3}{2N^2}\bigl[(1+N/q)^{2/3}-1\bigr]
 \lesssim N^{-4/3}q^{-2/3},
\]
because $q\le N$.  On an arc meeting $[0,T]$ with $a\ge1$, one has
$qT\ge1-N^{-1}\ge3/4$ and
$a\le qT+N^{-1}\le(4/3)qT$.  Therefore
\[
 \sum_{1\le a\le CqT}a^{-2/3}
 \le C(qT)^{1/3}.
\]
Multiplication by $q^{-1}q^{2/3}$ leaves exactly $CT^{1/3}$ for each
denominator $q$.  Summing the at most $N$ denominators yields
$CN^{1/3}T^{1/3}$ after the prefactor $N^{-2/3}$ is restored.

We may drop the coprimality condition when summing upper bounds. The total contribution of the nonzero arcs is at most
\begin{align*}
 C N^{-2/3}\sum_{q\le N}q^{-1}
       \sum_{1\le a\le CqT}(q/a)^{2/3}
 &\lesssim N^{-2/3}\sum_{q\le N}T^{1/3}\\
 &\lesssim (NT)^{1/3}.
\end{align*}
The inner sum is understood as zero if its index set is empty. This proves \eqref{eq:J-bound}, including $T$ shorter than the smallest nonzero rational arcs.

For the weighted integral, assume first $L\ge2$.  The segment
$0<\tau\le L^{-1}$ contributes at most $J_N(L^{-1})$ (with the endpoint
truncated at $1/2$).  On
\[
 E_j=\{2^j/L\le\tau\le2^{j+1}/L\}\cap(0,1/2],\qquad j\ge0,
\]
the weight is bounded by $C_K2^{-Kj}$.  Positivity of the integrand and
the initial-segment estimate give
\begin{align*}
 \int_{E_j}\tau^{-2/3}(1+L\tau)^{-K}A_N(\tau)^{2/3}\dd\tau
 &\le C_K2^{-Kj}
 J_N\bigl(\min\{2^{j+1}/L,1/2\}\bigr)\\
 &\lesssim_K2^{-Kj}\log(2N)
   +(N/L)^{1/3}2^{-(K-1/3)j}.
\end{align*}
Both geometric series converge for $K>1/3$.  If $1\le L<2$, the
unweighted bound with $T=1/2$ gives the same conclusion.  Thus no lower
bound on the observation length, and no disjointness of the rational
arcs, is required.
\end{proof}

\section{Two-frequency estimates}

\subsection{The duality estimate and its functional setting}

Return to the physical variables $t$ and $\theta$. A periodic cutoff of bandwidth $N$ has propagator kernel
\[
 K_N^{\T}(t,\theta)=\frac1{2\pi}\sum_n\chi(n/N)^2e^{in\theta-itn^2}.
\]
The changes of variables $\tau=t/(2\pi)$ and $y=\theta/(2\pi)$ identify this with \eqref{eq:W-def}, up to a constant. In particular, the time range $0<t\le1$ lies inside the range used in Lemma~\ref{lem:arithmetic}.

\begin{proposition}\label{prop:block}
Suppose that $h\in L^2(M)$ has hyperbolic spectral support in $[L/2,2L]$ and periodic Fourier support in $|n|\le c_0N$. The hyperbolic low-frequency support $[0,2]$ is permitted when $L=1$. For every interval $I$ of length at most one,
\begin{equation}\label{eq:block-log}
 \norm[{\X(I)}]{\U(t)h}
 \le C\left[\log(2N)+(N/L)^{1/3}\right]^{1/2}\norm[{2}]{h}.
\end{equation}
For $L\ge2$, one also has
\begin{equation}\label{eq:block-tail}
 \norm[{\X(I)}]{\U(t)h}\le C N^{1/3}L^{-1/6}\norm[{2}]{h}.
\end{equation}
Constants may depend on $c_0$, but not on $L,N,I$.
\end{proposition}

\begin{proof}
By time translation, it is enough to consider $I\subset[0,1]$. Choose bounded real smooth multipliers $B_L$ and $C_N$ equal to one on the indicated spectral supports and supported in fixed enlargements. Define
\[
 Tf(t)=\mathbf 1_I(t)\U(t)B_LC_Nf.
\]
The $TT^*$ kernel is $\U(t-s)B_L^2C_N^2$. Its $L^2\to L^2$ norm is bounded uniformly. By the product structure and Lemma~\ref{lem:H-kernel}, its $L^1\to L^\infty$ norm is at most
\[
 C_K|t-s|^{-1}(1+L|t-s|)^{-K}
 A_N\bigl((t-s)/(2\pi)\bigr).
\]
Riesz--Thorin interpolation between these two bounds gives an $L^{6/5}\to L^6$ norm bounded by
\begin{equation}\label{eq:TTkernel}
 C_K|t-s|^{-2/3}(1+L|t-s|)^{-K}
 A_N\bigl((t-s)/(2\pi)\bigr)^{2/3}.
\end{equation}
The interpolation parameter is $2/3$, since
\[
 \frac56=\frac{1-2/3}{2}+\frac23,
 \qquad\frac16=\frac{1-2/3}{2}.
\]
Here the value of $K$ has been renamed. For negative time the same bound holds by conjugation of the kernels.

To make the time argument explicit, extend $F\in\X'(I)$ by zero outside $I$ and set $a(s)=\|F(s)\|_{6/5}$. Write $b_{L,N}(v)$ for the right side of \eqref{eq:TTkernel}, multiplied by $\mathbf1_{\{|v|\le1\}}$. For smooth compactly supported test functions, the kernel formula and Minkowski give
\[
 \|TT^*F(t)\|_6\le\int_\R b_{L,N}(t-s)a(s)\dd s.
\]
Thus Young's convolution inequality yields
\begin{equation}\label{eq:young-TT-detail}
 \|TT^*F\|_{\X(I)}\le\|b_{L,N}\|_{L^1(\R)}\|F\|_{\X'(I)}.
\end{equation} For $N\ge2$, Lemma~\ref{lem:arithmetic}, with the fixed time rescaling, yields
\[
 \norm[{\X'(I)\to\X(I)}]{TT^*}
 \lesssim\log(2N)+(N/L)^{1/3}.
\]
The usual Hilbert-space duality argument gives the square root of this bound for $T:L^2\to\X(I)$. Explicitly,
\[
 \norm[{2}]{T^*F}^2
 =\langle TT^*F,F\rangle
 \le\norm[{\X}]{TT^*F}\norm[{\X'}]{F}.
\]
Here the pairing is the space-time integral, conjugate-linear in the second variable. On the dense class $L^1(I;L^2(M))\cap\X'(I)$ the adjoint is the Bochner integral
\[
 T^*F=B_LC_N\int_I\U(-s)F(s)\dd s.
\]
The preceding inequality bounds this adjoint into $L^2$ by the square root of $\|b_{L,N}\|_1$, so it extends to all of $\X'(I)$. Since $\X(I)$ is reflexive and its dual is $\X'(I)$, duality then defines the desired bounded operator $T$. This also justifies the argument without assuming the Strichartz bound in advance. Its action agrees with the propagator on the initial dense class and hence in distributions. Finally $B_LC_Nh=h$, proving \eqref{eq:block-log}. The general $TT^*$ principle is part of the standard Strichartz framework \cite{KT}; the present calculation uses an integrable time majorant and the elementary form \eqref{eq:young-TT-detail}.

When $N=1$, the same conclusion follows directly from $A_1\lesssim1$ and the integrability of \eqref{eq:TTkernel}. For the second estimate, use $A_N\lesssim N$ in \eqref{eq:TTkernel}. Since
\[
 \int_0^1t^{-2/3}(1+Lt)^{-K}\dd t\lesssim L^{-1/3}
\]
for $K>1/3$, the squared operator norm is $O(N^{2/3}L^{-1/3})$. This proves \eqref{eq:block-tail}.
\end{proof}

\subsection{Translation of the periodic frequency interval}

\begin{corollary}[Matched bandwidth]\label{cor:matched}
Suppose $h$ has hyperbolic frequency $L$ as in Proposition~\ref{prop:block}, and periodic Fourier support in an interval $J$ of length at most $c_0L$, with arbitrary center. Then
\begin{equation}\label{eq:matched}
 \norm[{\X(I)}]{\U(t)h}\lesssim\sqrt{\log(2L)}\norm[{2}]{h}
\end{equation}
for every interval $I$ of length at most one.
\end{corollary}

\begin{proof}
Choose an integer $n_0$ within distance one of the center of $J$. Writing $h=e^{in_0\theta}g$, one has
\begin{equation}\label{eq:galilean}
 \U(t)h(x,\theta)=e^{in_0\theta-itn_0^2}
 \U(t)g(x,\theta-2n_0t).
\end{equation}
The change in $\theta$ is a measure-preserving translation for each $t$. Apply \eqref{eq:block-log} to $g$ with periodic scale $N=L$, increasing the fixed bandwidth constant if necessary.
\end{proof}

\subsection{The direct one-third derivative estimate}\label{sec:baseline}
There is a shorter argument with a larger regularity cost.  The
one-dimensional Sobolev embedding
$H^{1/3}(\T)\hookrightarrow L^6(\T)$ gives, for
$f(x,\theta)=\sum_n f_n(x)e^{in\theta}$,
\[
 \|\U(t)f\|_{L^6(M)}
 \lesssim
 \left\|\left(\sum_n\br n^{2/3}
       |e^{it\Delta_H}f_n|^2\right)^{1/2}\right\|_{L_x^6(\HH)}.
\]
The factors $e^{-itn^2}$ have modulus one and disappear from this
inequality.  Minkowski and the hyperbolic Strichartz estimate
$\|e^{it\Delta_H}g\|_{L_t^2([0,1];L_x^6)}\lesssim\|g\|_2$
from \cite{AP} yield
\begin{align*}
 \|\U(t)f\|_{\X([0,1])}
 &\lesssim
 \left(\sum_n\br n^{2/3}
 \|e^{it\Delta_H}f_n\|_{L_t^2L_x^6}^2\right)^{1/2}\\
 &\lesssim
 \left(\sum_n\br n^{2/3}\|f_n\|_2^2\right)^{1/2}
 \asymp\|\br{D_\theta}^{1/3}f\|_2.
\end{align*}
This argument does not use the periodic time oscillations.  The reduction
from $1/3$ to every positive regularity requires the frequency-dependent
time interaction and the decoupling step developed in the following sections.

\section{Decoupling with an additional spatial variable}

This section records the extension needed to apply parabola decoupling without interchanging the time and hyperbolic-space norms. Let $(Z,\mu)$ be a sigma-finite measure space. All Fourier support conditions in this section refer only to the displayed Euclidean variables $(t,y)$; the variable $z\in Z$ is not transformed.

For $0<\delta\le1$, a center $\omega$ in a fixed bounded interval, and a fixed $A\ge1$, define the parallelogram
\begin{equation}\label{eq:cap}
 \Theta_\omega(\delta)=
 \{(\tau,\xi):|\xi-\omega|\le A\delta,
 \ |\tau+2\omega\xi-\omega^2|\le A\delta^2\}.
\end{equation}
Centers will be $c_A\delta$ separated, with $c_A$ sufficiently large. A family with spacing comparable to $\delta$ can be divided into a bounded number of such subfamilies.

\begin{proposition}[Mixed decoupling with an $L^6$ variable]\label{prop:vector-dec}
Let $F_\omega(t,y,z)$ be a finite family whose $(t,y)$ Fourier supports lie in $\Theta_\omega(\delta)$. Then, for every $\eps>0$,
\begin{equation}\label{eq:vector-dec}
 \norm[{L_t^2L_{y,z}^6}]{\sum_\omega F_\omega}
 \le C_{A,\eps}\delta^{-\eps}
 \left(\sum_\omega\norm[{L_t^2L_{y,z}^6}]{F_\omega}^2\right)^{1/2}.
\end{equation}
The constant is independent of $Z$ and of the number of functions.
\end{proposition}

\begin{proof}
We first explain the pure $L^6$ estimate. The scalar parabola decoupling theorem \cite{BD} gives, for functions supported in these caps or fixed enlargements of them,
\[
 \norm[{L^6_{t,y}}]{\sum_\omega G_\omega(\cdot,\cdot,z)}
 \le C_{A,\eps}\delta^{-\eps}
 \left(\sum_\omega\norm[{L^6_{t,y}}]{G_\omega(\cdot,\cdot,z)}^2\right)^{1/2}.
\]
This is the standard global family form of decoupling; see also \cite[Theorem 1, with $(q,r)=(6,6)$]{DJLM}. To obtain it from the local weighted theorem, cover $\R^2$ by boxes $B$ of side $\delta^{-2}$, with the usual rapidly decreasing weights $w_B$ satisfying $\sum_Bw_B\lesssim1$. If $a_{\omega,B}=\|G_\omega\|_{L^6(w_B)}$, the required summation is
\[
 \left(\sum_B\left(\sum_\omega a_{\omega,B}^2\right)^3\right)^{1/6}
 \le\left(\sum_\omega\left(\sum_Ba_{\omega,B}^6\right)^{1/3}\right)^{1/2}
 \lesssim\left(\sum_\omega\|G_\omega\|_6^2\right)^{1/2}.
\]
Integrating the scalar global estimate in $z$ and using Minkowski once more yields
\begin{equation}\label{eq:vector-pure6}
 \norm[{L^6_{t,y,z}}]{\sum_\omega G_\omega}
 \le C_{A,\eps}\delta^{-\eps}
 \left(\sum_\omega\norm[{L^6_{t,y,z}}]{G_\omega}^2\right)^{1/2}.
\end{equation}
Indeed, the required inequality is $\| (\sum a_\omega^2)^{1/2}\|_{L_z^6}
\le(\sum\|a_\omega\|_{L_z^6}^2)^{1/2}$.

We now reduce the time exponent, following the reverse-H\"older mechanism in \cite[Section 3.4]{DJLM}, with its action on the extra variable made explicit. Choose $\phi\in\Sch(\R)$ with compact Fourier support and $|\phi|\ge1$ on $[-1,1]$. Such a function is obtained, for example, by dilating and multiplying the squared modulus of the inverse Fourier transform of a smooth nonnegative bump supported near zero. Partition the time line into intervals $I$ of length
\[
 T_\delta=\delta^{-2},
\]
with centers $c_I$, and let $\phi_I(t)=\phi((t-c_I)/T_\delta)$. The rapid decay of $\phi$ gives
\begin{equation}\label{eq:time-weight-sum}
 \sup_t\sum_I|\phi_I(t)|^2<\infty.
\end{equation}

Let $B_\phi$ satisfy
$\supp\widehat\phi\subset[-B_\phi,B_\phi]$.
For a fixed time interval $I$, the function
$G_\omega=\phi_I F_\omega$ has partial Fourier support in
\[
 |\xi-\omega|\le A\delta,\qquad
 |\tau+2\omega\xi-\omega^2|
       \le(A+B_\phi)\delta^2.
\]
Choose a fixed $C>A+B_\phi+A+1$ and smooth compactly supported
functions $\kappa,\beta$ equal to one on $[-1,1]$. The multiplier
\[
 m_\omega(\tau,\xi)=
 \kappa\bigl((\xi-\omega)/(C\delta)\bigr)
 \beta\bigl((\tau+2\omega\xi-\omega^2)/(C\delta^2)\bigr)
\]
equals one on that support. Write
$k=\mathcal F^{-1}\kappa$ and $b=\mathcal F^{-1}\beta$, with the
one-dimensional inverse transform normalized by $(2\pi)^{-1}$.
Using the coordinates
\[
 u=\frac{\xi-\omega}{C\delta},\qquad
 v=\frac{\tau+2\omega\xi-\omega^2}{C\delta^2},
\]
whose inverse is
\[
 \xi=\omega+C\delta u,\qquad
 \tau=-\omega^2-2\omega C\delta u+C\delta^2v,
\]
and whose Jacobian has absolute value $C^2\delta^3$, gives
\begin{equation}\label{eq:reproducer}
 K_\omega(t,y)=C^2\delta^3
 e^{i(\omega y-\omega^2t)}
 k\bigl(C\delta(y-2\omega t)\bigr)
 b(C\delta^2t).
\end{equation}
In particular, if
$a_\omega(t)=\int_\R|K_\omega(t,y)|\dd y$, then
\[
 a_\omega(t)
 =C\delta^2\|k\|_{L^1(\R)}|b(C\delta^2t)|,
\]
and therefore
\begin{equation}\label{eq:kernel-young}
 \|a_\omega\|_{L^{3/2}(\R)}
 =(C\delta^2)^{1/3}\|k\|_1\|b\|_{3/2}
 \lesssim_A\delta^{2/3}.
\end{equation}
The estimate is uniform in $\omega$: the shear
$y\mapsto y-2\omega t$ is a translation at each fixed time.

For a function $G$ with finite $L_t^2L_{y,z}^6$ norm, Minkowski and
translation invariance give
\[
 \|(K_\omega*_{t,y}G)(t)\|_{L_{y,z}^6}
 \le\int_\R a_\omega(t-s)\|G(s)\|_{L_{y,z}^6}\dd s.
\]
Scalar Young's inequality, with
$1+1/6=1/(3/2)+1/2$, now implies
\[
 \|K_\omega*_{t,y}G\|_{L_t^6L_{y,z}^6}
 \le\|a_\omega\|_{3/2}\|G\|_{L_t^2L_{y,z}^6}
 \lesssim_A\delta^{2/3}\|G\|_{L_t^2L_{y,z}^6}.
\]
Apply this with $G=G_\omega$, for which the Fourier support identity
gives $K_\omega*_{t,y}G_\omega=G_\omega$. In particular,
\begin{equation}\label{eq:reverse-holder}
 \norm[{L_t^6L_{y,z}^6}]{G_\omega}
 \lesssim_A\delta^{2/3}\norm[{L_t^2L_{y,z}^6}]{G_\omega}.
\end{equation}
This proves the reverse
H\"older estimate with the additional spatial variable included.

H\"older on $I$, \eqref{eq:vector-pure6}, and \eqref{eq:reverse-holder} now give
\begin{align}\label{eq:local-mixed-dec}
 \norm[{L_t^2(I;L_{y,z}^6)}]{\sum_\omega F_\omega}
 &\le T_\delta^{1/3}\norm[{L^6_{t,y,z}}]{\sum_\omega\phi_I F_\omega}\\
 &\le C_{A,\eps}\delta^{-2/3-\eps}
 \left(\sum_\omega\norm[{L^6_{t,y,z}}]{\phi_I F_\omega}^2\right)^{1/2}\notag\\
 &\le C_{A,\eps}\delta^{-\eps}
 \left(\sum_\omega\norm[{L_t^2L_{y,z}^6}]{\phi_I F_\omega}^2\right)^{1/2}.\notag
\end{align}
The final summation is worth writing explicitly, because it preserves the order of the mixed norms. Squaring \eqref{eq:local-mixed-dec} and summing over the disjoint intervals gives
\begin{align*}
 \Bigl\|\sum_\omega F_\omega\Bigr\|_{L_t^2L_{y,z}^6}^2
 &\le C_{A,\eps}\delta^{-2\eps}
      \sum_I\sum_\omega\int_\R |\phi_I(t)|^2
               \|F_\omega(t)\|_{L_{y,z}^6}^2\dd t\\
 &\le C_{A,\eps}\delta^{-2\eps}
      \sum_\omega\|F_\omega\|_{L_t^2L_{y,z}^6}^2,
\end{align*}
where Tonelli and \eqref{eq:time-weight-sum} were used in the last line. The time weight is the same function $\phi_I$ before and after the reverse H\"older estimate. This proves \eqref{eq:vector-dec}. Initially one may work with finite sums and smooth integrable functions; approximation gives the statement whenever the right side is finite. Boundedly overlapping cap families are handled by coloring and the triangle inequality.
\end{proof}

\subsection{Periodic transference}

We use the Fourier transform
\[
 \widehat v(\xi)=\int_{\R}e^{-iy\xi}v(y)\dd y,
 \qquad
 v(y)=\frac1{2\pi}\int_{\R}e^{iy\xi}\widehat v(\xi)\dd\xi.
\]
The compact Fourier support of a spatial cutoff gives an exact
periodization identity. In particular, the following argument does not
require passing to a limit in the mixed norm.

\begin{lemma}[Exact periodic averaging]\label{lem:exact-periodic-average}
Let $\zeta\in\Sch(\R)$ be nonzero, with
$\supp\widehat\zeta\subset[-1,1]$, and put
\[
 \zeta_R(y)=\zeta(y/R),\qquad
 c_\zeta=\frac1{2\pi}\int_\R|\zeta(v)|^6\dd v>0.
\]
For every $R>6$ and every measurable $2\pi$-periodic function
$H(y,z)$,
\begin{equation}\label{eq:exact-periodic-sixth}
 \norm[{L^6(\R_y\times Z)}]{\zeta_RH}^{6}
 =R c_\zeta\,\norm[{L^6(\T\times Z)}]{H}^{6}.
\end{equation}
The identity is valid also as an identity of extended nonnegative
numbers. Consequently, if $H=H(t,y,z)$ is periodic in $y$, then
\begin{equation}\label{eq:exact-periodic-mixed}
 \norm[{L_t^2L_{y,z}^6(\R\times\R\times Z)}]{\zeta_RH}
 =(R c_\zeta)^{1/6}
 \norm[{L_t^2L_{\theta,z}^6(\R\times\T\times Z)}]{H}.
\end{equation}
\end{lemma}

\begin{proof}
Set $w=|\zeta|^6=\zeta^3\overline\zeta^{\,3}$. Then $w\in\Sch(\R)$
and $\supp\widehat w\subset[-6,6]$. The smooth periodization
\[
 W_R(\vartheta)=\frac1R\sum_{k\in\Z}
       w\bigl((\vartheta+2\pi k)/R\bigr)
\]
has Fourier coefficients
\[
 \frac1{2\pi}\int_0^{2\pi}
       W_R(\vartheta)e^{-im\vartheta}\dd\vartheta
 =\frac1{2\pi}\widehat w(mR),\qquad m\in\Z.
\]
These formulas follow by absolute convergence of the Schwartz sum and
the substitution $\vartheta+2\pi k=Rv$. If $R>6$, all nonzero Fourier
coefficients vanish. Hence
\[
 W_R(\vartheta)=\frac{\widehat w(0)}{2\pi}=c_\zeta
 \quad\text{for all }\vartheta.
\]
Tonelli's theorem and periodicity now give
\begin{align*}
 \frac1R\int_{\R}\int_Z
       w(y/R)|H(y,z)|^6\dd\mu(z)\dd y
 &=\int_0^{2\pi}\int_Z
       W_R(\vartheta)|H(\vartheta,z)|^6
          \dd\mu(z)\dd\vartheta\\
 &=c_\zeta\norm[{L^6(\T\times Z)}]{H}^6.
\end{align*}
This proves \eqref{eq:exact-periodic-sixth}. Take the one-third power
of that identity at each time and integrate in $t$ to obtain
\eqref{eq:exact-periodic-mixed}.
\end{proof}

\begin{proposition}\label{prop:periodic-dec}
Let $1\le L\le N$, and let $F(t,\theta,z)$ be periodic in $\theta$,
with partial Fourier support contained in
\begin{equation}\label{eq:thick-parabola}
 \{(\tau,n):|n|\le C_0N,\quad|\tau+n^2|\le C_0L^2\}.
\end{equation}
Choose a standard smooth partition of unity in the $n$ variable whose
pieces have support in intervals $J$ of length comparable to $L$,
whose centers are separated by a fixed positive multiple of $L$,
and whose overlap and multiplier sizes are uniformly bounded. Write
$F=\sum_JF_J$. Then, for every $\eps>0$,
\begin{equation}\label{eq:periodic-dec}
 \norm[{L_t^2L_{\theta,z}^6}]{F}
 \le C_{C_0,\eps}(N/L)^\eps
       \left(\sum_J\norm[{L_t^2L_{\theta,z}^6}]{F_J}^2\right)^{1/2}.
\end{equation}
All time norms in this proposition are over $\R$. The same assertion
holds for any finite family with the corresponding interval support
conditions and with sum $F$.
\end{proposition}

\begin{proof}
First let the spatial variable be real. Suppose that
$F_J(t,y,z)$ has $y$ frequency in an interval of width $O(L)$ and
center $n_J$, and that
$|\tau+\xi^2|\le C_0L^2$ on its Fourier support. Define the rescaled
functions
\[
 \widetilde F_J(s,v,z)=F_J(s/N^2,v/N,z),
 \qquad
 \delta=L/N,\qquad\omega_J=n_J/N.
\]
The partial Fourier transform obeys
\[
 \widehat{\widetilde F_J}(\sigma,\xi,z)
 =N^3\widehat F_J(N^2\sigma,N\xi,z).
\]
Thus the normalized frequencies satisfy
\[
 |\xi-\omega_J|\lesssim\delta,\qquad
 |\sigma+\xi^2|\lesssim\delta^2,
\]
and consequently
\[
 |\sigma+2\omega_J\xi-\omega_J^2|
 \le|\sigma+\xi^2|+|\xi-\omega_J|^2
 \lesssim\delta^2.
\]
The centers $\omega_J$ remain in a fixed bounded interval and have
spacing comparable to $\delta$. Divide them into a bounded number of
subfamilies with the separation required in
Proposition~\ref{prop:vector-dec}. Apply that proposition to each
subfamily and then use the triangle inequality and Cauchy--Schwarz
over the finitely many colors.

For clarity, the exact mixed-norm scaling is
\begin{align*}
 \norm[{L_s^2L_{v,z}^6}]{\widetilde F_J}^{\,2}
 &=\int_\R
       \left(N\int_{\R\times Z}
            |F_J(s/N^2,y,z)|^6\dd y\dd\mu(z)\right)^{1/3}
          \dd s\\
 &=N^{7/3}\norm[{L_t^2L_{y,z}^6}]{F_J}^{\,2}.
\end{align*}
Therefore every norm in the rescaled inequality is multiplied by the
same factor $N^{7/6}$. Cancelling it proves the real-line family
inequality with constant $C_\eps(N/L)^\eps$. If $\delta$ is bounded
below, there are only a bounded number of intervals and the triangle
inequality gives the same conclusion directly. Fixed enlargements of
the support constants affect only the implicit constant.

Now return to periodic $F_J$. It suffices to consider a finite family
for which the right side of \eqref{eq:periodic-dec} is finite;
otherwise the assertion is immediate. Lift these functions
periodically to the real $y$ line, and choose $\zeta$ as in
Lemma~\ref{lem:exact-periodic-average}. Multiplication by $\zeta_R$
replaces a discrete frequency $n$ by
$\xi=n+\nu$, where $|\nu|\le R^{-1}$, and leaves the time frequency
unchanged. Hence on the Fourier support of $\zeta_RF_J$,
\[
 |\tau+\xi^2|
 \le C_0L^2+2C_0N/R+R^{-2}.
\]
Choose any finite
\[
 R\ge C(1+N/L^2),\qquad R>6,
\]
with $C$ sufficiently large in terms of the fixed support constants.
Since $L\ge1$, the spatial interval width increases by at most
$R^{-1}\le L$, and the last display is bounded by a fixed multiple
of $L^2$. Thus $\{\zeta_RF_J\}_J$ satisfies the real-line family
estimate just proved, with constants independent of $R$. This uses
the functions themselves, so no commutator between a frequency
projection and $\zeta_R$ occurs.

Apply that estimate and use
\eqref{eq:exact-periodic-mixed} on its left side and on each term of
its right side:
\begin{align*}
 (Rc_\zeta)^{1/6}\norm[{L_t^2L_{\theta,z}^6}]{F}
 &\le C_{C_0,\eps}(N/L)^\eps
     \left(\sum_J
       (Rc_\zeta)^{1/3}
       \norm[{L_t^2L_{\theta,z}^6}]{F_J}^2\right)^{1/2}.
\end{align*}
Cancel the common positive factor $(Rc_\zeta)^{1/6}$ to obtain
\eqref{eq:periodic-dec}. Only finitely many intervals meet the
frequency range in \eqref{eq:thick-parabola}; the asserted partition
form is therefore included in the finite-family statement.
\end{proof}

\section{The anisotropic estimate}

\subsection{Time localization and low hyperbolic frequencies}

\begin{proposition}\label{prop:lowH}
Let $1\le L\le N$ and let $h$ satisfy the spectral support assumptions of Proposition~\ref{prop:block}. For every $\eps>0$,
\begin{equation}\label{eq:lowH}
 \norm[{\X([0,1])}]{\U(t)h}
 \le C_\eps(N/L)^\eps\sqrt{\log(2L)}\norm[{2}]{h}.
\end{equation}
\end{proposition}

\begin{proof}
Choose $\eta\in\Sch(\R)$ with compact Fourier support and $|\eta|\ge1$ on $[0,1]$. For
\[
 F(t,x,\theta)=\eta(t)\U(t)h(x,\theta),
\]
the spectral theorem and \eqref{eq:product} imply
\begin{equation}\label{eq:modulation}
 \supp\mathcal F_{t,\theta}F
 \subset\{(\tau,n):|n|\le C_0N,\ |\tau+n^2|\le C_1L^2\}.
\end{equation}
The constant $1/4$ and the bounded support of $\widehat\eta$ are absorbed by $C_1L^2$, since $L\ge1$. Here is a spectral-theorem justification of this partial Fourier support. Let $E_H$ be the projection-valued resolution of $\DHyper$ and let $h_n$ be the $n$th Fourier coefficient of $h$. The time transform is the $L^2(\HH)$-valued expression
\begin{equation}\label{eq:modulation-spectral-integral}
 \widehat F_n(\tau)
 =\int_0^\infty
       \widehat\eta(\tau+n^2+\lambda^2+1/4)\dd E_H(\lambda)h_n.
\end{equation}
One obtains this formula first after pairing with any $L^2$ test function in $x$, using the bounded spectral support and $\eta\in L^1(\R)$ to interchange the integrations. The integrand vanishes unless
$|\tau+n^2+\lambda^2+1/4|\le B_\eta$, where $B_\eta$ is fixed. Since $\lambda\le2L$ on the support of $h$, it follows that
\[
 |\tau+n^2|\le4L^2+1/4+B_\eta\le C_1L^2.
\]
This proves \eqref{eq:modulation} distributionally, including $L=1$. The same support statement is valid in the $L^6$ spectator formulation: the weighted space-time norms are finite by Proposition~\ref{prop:block}, and pairing against compactly supported tests in $x$ identifies the two distributions.

For an explicit interval partition, choose $\rho\in C_c^\infty(\R)$ with $\rho\ge0$ and $\rho>0$ on $[-1,1]$, and put
\[
 \rho_j(v)=\frac{\rho(v-j)}{\sum_{k\in\Z}\rho(v-k)},\qquad
 h_{J_j}=\rho_j(D_\theta/L)h.
\]
The denominator is smooth, one-periodic, and bounded away from zero. These multipliers sum to one, have uniformly bounded overlap, and are supported in intervals $J_j$ of width $O(L)$ centered at $jL$. Only $O(1+N/L)$ pieces act nontrivially on $h$. We abbreviate them by $h_J$. Then $F_J=\eta\U(t)h_J$. Proposition~\ref{prop:periodic-dec}, with $Z=\HH$, gives
\begin{equation}\label{eq:apply-dec}
 \norm[{\X([0,1])}]{\U(t)h}
 \le C_\eps(N/L)^\eps
 \left(\sum_J\norm[{L_t^2(\R;L^6(M))}]{\eta\U(t)h_J}^2\right)^{1/2}.
\end{equation}
These weighted norms are finite by Proposition~\ref{prop:block} and rapid decay of $\eta$, so the transference argument is applicable.

By Corollary~\ref{cor:matched}, on every translated unit interval $[k,k+1]$,
\[
 \norm[{\X([k,k+1])}]{\U(t)h_J}
 \lesssim\sqrt{\log(2L)}\norm[{2}]{h_J}.
\]
Indeed, time translation replaces $h_J$ by $\U(k)h_J$, which preserves both its $L^2$ norm and its spectral supports. Hence
\begin{align}\label{eq:eta-bound}
 \norm[{L_t^2(\R;L^6)}]{\eta\U(t)h_J}^2
 &\le\sum_{k\in\Z}\sup_{[k,k+1]}|\eta|^2
       \norm[{\X([k,k+1])}]{\U(t)h_J}^2\\
 &\lesssim\log(2L)\norm[{2}]{h_J}^2.\notag
\end{align}
Finally, Parseval and bounded overlap give
\[
 \sum_j\|h_{J_j}\|_2^2
 =2\pi\sum_{n\in\Z}\sum_j|\rho_j(n/L)|^2\|h_n\|_{L_x^2}^2
 \le C\|h\|_2^2.
\] Insert this into \eqref{eq:apply-dec} and \eqref{eq:eta-bound}.
\end{proof}

\subsection{Summation in the hyperbolic frequency}

\begin{proposition}[One periodic dyadic band]\label{prop:N}
For every $\eps>0$ and $N\in\mathcal D$,
\begin{equation}\label{eq:N-band}
 \norm[{\X([0,1])}]{\U(t)Q_Nf}
 \le C_\eps N^\eps\norm[{2}]{Q_Nf}.
\end{equation}
\end{proposition}

\begin{proof}
Put $h_L=P_LQ_Nf$. For finite hyperbolic sums, the triangle inequality and Cauchy--Schwarz yield
\begin{equation}\label{eq:H-sum}
 \norm[{\X}]{\sum_L\U(t)h_L}
 \le\left(\sum_L C_{L,N}^2\right)^{1/2}
       \left(\sum_L\norm[{2}]{h_L}^2\right)^{1/2},
\end{equation}
where the block constants may be chosen as follows:
\begin{equation}\label{eq:three-C}
 C_{L,N}\lesssim_\eps
 \begin{cases}
 (N/L)^\eps\sqrt{\log(2L)},&L\le N,\\
 \sqrt{\log(2N)},&N<L\le N^2,\\
 N^{1/3}L^{-1/6},&L>N^2.
 \end{cases}
\end{equation}
The first line is Proposition~\ref{prop:lowH}. The second follows from \eqref{eq:block-log}; the third is \eqref{eq:block-tail}. Empty ranges are omitted.

All sums over $L$ below are dyadic. For the first range,
\begin{equation}\label{eq:sum-low}
 \sum_{L\le N}C_{L,N}^2
 \lesssim_\eps N^{2\eps}
 \sum_{j=0}^\infty2^{-2\eps j}(1+j)
 \lesssim_\eps N^{2\eps}.
\end{equation}
There are $O(\log(2N))$ terms in the middle range, so
\begin{equation}\label{eq:sum-mid}
 \sum_{N<L\le N^2}C_{L,N}^2\lesssim\log^2(2N).
\end{equation}
The tail is summable:
\begin{equation}\label{eq:sum-high}
 \sum_{L>N^2}C_{L,N}^2
 \lesssim N^{2/3}\sum_{L>N^2}L^{-1/3}\lesssim1.
\end{equation}
For completeness, if $N=2^m$ then the middle range contains at most $m$ indices and the last sum equals a constant multiple of
\[
 N^{2/3}\sum_{j>2m}2^{-j/3}
 =N^{2/3}\frac{2^{-(2m+1)/3}}{1-2^{-1/3}}\lesssim1.
\]
In the first range, with $r_\eps=2^{-2\eps}<1$, the sum is bounded by
\[
 \sum_{j\ge0}(1+j)r_\eps^j=(1-r_\eps)^{-2}.
\]
For each fixed $\eps>0$, also $\log^2(2N)\le C_\eps N^{2\eps}$. Consequently
\begin{equation}\label{eq:constant-square-sum}
 \sum_{L\in\mathcal D}C_{L,N}^2\le C_\eps N^{2\eps}.
\end{equation}
Together with \eqref{eq:H-sum} and \eqref{eq:H-L2}, this proves \eqref{eq:N-band} for finite sums. The case $N=1$ is included: the first range consists only of $L=1$, the middle range is empty, and the same tail converges.

Here is the passage to the infinite series. If $\mathcal E\subset\mathcal D$ is a finite set contained in $\{L>R\}$, Cauchy--Schwarz gives
\[
 \Bigl\|\sum_{L\in\mathcal E}\U(t)P_LQ_Nf\Bigr\|_\X
 \le\Bigl(\sum_{L>R}C_{L,N}^2\Bigr)^{1/2}
       \Bigl(\sum_{L>R}\|P_LQ_Nf\|_2^2\Bigr)^{1/2}.
\]
This tends to zero as $R\to\infty$. Completeness of $\X$ provides a limit. The same spectral partial sums converge to $\U(t)Q_Nf$ in $L_t^\infty L^2$, by unitarity. Pairing against a compactly supported smooth space-time test function identifies both limits as distributions, so the $\X$ limit is precisely the original evolution.
\end{proof}

\subsection{Periodic reconstruction and completion of the proof}

\begin{proof}[Proof of Theorem~\ref{thm:main}]
For a trigonometric polynomial $v(x,\theta)$, the scalar Littlewood--Paley theorem on $\T$ gives, for almost every $x$,
\[
 \|v(x,\cdot)\|_{L_\theta^6}
 \lesssim\Bigl\|\Bigl(\sum_N|Q_Nv(x,\cdot)|^2\Bigr)^{1/2}\Bigr\|_{L_\theta^6}.
\]
The constant is independent of $x$. Taking sixth powers and integrating in $x$, then taking the $L_t^2$ norm, is legitimate by Tonelli. Since $Q_N$ commutes with $\U(t)$, this yields the first line below. The second line uses only a triangle inequality:
\begin{align}\label{eq:theta-sum}
 \norm[{\X}]{\U(t)f}
 &\lesssim\norm[{L_t^2L_{x,\theta}^6}]
       {\left(\sum_N|\U(t)Q_Nf|^2\right)^{1/2}}\\
 &\le\left(\sum_N\norm[{\X}]{\U(t)Q_Nf}^2\right)^{1/2}.\notag
\end{align}
Proposition~\ref{prop:N}, with $\eps=s$, gives
\[
 \norm[{\X}]{\U(t)f}
 \lesssim_s\left(\sum_NN^{2s}\norm[{2}]{Q_Nf}^2\right)^{1/2}
 \lesssim_s\norm[{2}]{\br{D_\theta}^{s}f}.
\]
The last step is exactly \eqref{eq:theta-weighted-parseval}. For clarity about the norm order, at each fixed time one has
\[
 \Bigl\|\Bigl(\sum_N|u_N(t)|^2\Bigr)^{1/2}\Bigr\|_6^2
 =\Bigl\|\sum_N|u_N(t)|^2\Bigr\|_3
 \le\sum_N\|u_N(t)\|_6^2.
\]
Integrating this inequality in time gives \eqref{eq:theta-sum} with the required $L_t^2L_{x,\theta}^6$ order. It does not exchange these two mixed norms.

We first apply the estimate to data with finitely many periodic and hyperbolic bands. Such data are dense in the norm \eqref{eq:data-norm}: spectral truncations converge by dominated convergence in the joint spectral measure. If $f^{(j)}\to f$ in that norm, the estimate makes $\U(t)f^{(j)}$ Cauchy in $\X$. Unitarity also gives convergence in $L_t^\infty L^2$ because $s>0$. Distributional uniqueness identifies the $\X$ limit with $\U(t)f$, and passing to the limit proves \eqref{eq:main} for every admissible datum. The elementary Fourier and square-function facts used here are recorded, for example, in \cite{Grafakos}.

To exclude negative regularity, choose a nonzero $\phi\in L^2(\HH)$ with compact hyperbolic spectral support and set $f_n(x,\theta)=\phi(x)e^{in\theta}$. One has the exact identities
\begin{align*}
 \|\U(t)f_n\|_{\X([0,1])}
 &=(2\pi)^{1/6}\|e^{it\Delta_H}\phi\|_{L_t^2([0,1];L_x^6)},\\
 \|\br{D_\theta}^sf_n\|_2
 &=(2\pi)^{1/2}\br n^s\|\phi\|_2.
\end{align*}
The first norm is finite: apply the matched-band estimate to $f_0$ and cover the compact hyperbolic support by finitely many bands. It is strictly positive, since otherwise $e^{it\Delta_H}\phi=0$ almost everywhere for almost every $t$, contradicting its constant nonzero $L^2$ norm. The first expression is independent of $n$, while the second tends to zero for $s<0$. Thus $\mathcal A\subset[0,\infty)$, completing the proof.
\end{proof}

\section{The zero-regularity endpoint}

Theorem~\ref{thm:main} identifies the infimum of the regularity indices but does not prove that this infimum is attained. There are two explicit places where the present proof permits a small frequency loss. Proposition~\ref{prop:vector-dec} uses a decoupling constant $C_\eps\delta^{-\eps}$, and the elementary hyperbolic reconstruction in \eqref{eq:H-sum} absorbs logarithmic sums into a positive power of the periodic frequency. Neither step gives a bound uniform as $\eps\downarrow0$.

The single-mode example in the proof of Theorem~\ref{thm:main} excludes $s<0$ but gives no obstruction at $s=0$. Likewise, a lower bound for an absolute propagator kernel would not by itself yield a lower bound for the homogeneous operator norm. The arithmetic estimate \eqref{eq:weighted-J} is used as an upper bound in $TT^*$, and need not be saturated by a common family of initial data.

One can state the remaining issue as a uniform frequency-band problem. Define
\begin{equation}\label{eq:endpoint-band-constants}
 \mathfrak C_N=
 \sup_{\substack{0\ne h\in\operatorname{Ran}Q_N}}
 \frac{\|\U(t)h\|_{\X([0,1])}}{\|h\|_2}.
\end{equation}
Proposition~\ref{prop:N} says that $\mathfrak C_N\le C_\eps N^\eps$ for every $\eps>0$. The zero-regularity estimate holds if and only if
\begin{equation}\label{eq:endpoint-equivalent}
 \sup_{N\in\mathcal D}\mathfrak C_N<\infty.
\end{equation}
Necessity follows by applying that estimate to data in a single band. Conversely, \eqref{eq:theta-sum}, \eqref{eq:endpoint-equivalent}, and \eqref{eq:theta-weighted-parseval} with $s=0$ give the estimate for finite frequency sums; the same density argument then gives it for all $L^2$ data.

The estimates in this paper do not imply \eqref{eq:endpoint-equivalent}. For example, the explicit summation constant in the low hyperbolic range contains
\[
 \sum_{j\ge0}(j+1)2^{-2\eps j}
 =(1-2^{-2\eps})^{-2}\asymp\eps^{-2}
 \quad\text{as }\eps\downarrow0.
\]
In addition, the decoupling theorem does not supply a constant uniform in $\eps$. The bounds $\mathfrak C_N\lesssim_\eps N^\eps$ exclude growth by a fixed positive power, but do not exclude logarithmic or other subpolynomial growth. No such lower growth is proved here.

Accordingly, the remaining endpoint problem is the precise estimate
\begin{equation}\label{eq:endpoint-question}
 \norm[{L^2([0,1];L^6(\HH\times\T))}]{\U(t)f}
 \stackrel{?}{\lesssim}\norm[{L^2(\HH\times\T)}]{f}.
\end{equation}
The result proved here rules out a strictly positive optimal power loss. It leaves open, within this argument, both a uniform endpoint estimate and a possible subpolynomial obstruction to it.

\appendix
\section{Uniform symbols in the spherical expansion}\label{app:spherical}

This appendix verifies the differentiated amplitude estimates in Lemma~\ref{lem:spherical}. The integral representations and Harish--Chandra connection coefficients are classical spherical-function identities with the normalization $\varphi_\lambda(0)=1$; see \cite[Chapter IV]{Helgason}. They are consistent with the spectral-measure description in \cite{CH} and the spherical-function analysis in \cite[Section 2]{GL}. The hypergeometric transformation and gamma identities used below are listed in \cite[\S\S5.5, 5.11, 15.8]{DLMF}. The issue addressed here is uniformity of arbitrarily many derivatives in $\lambda$.

\subsection{Distances at most one}

The Mehler--Dirichlet formula is
\begin{equation}\label{eq:mehler}
 \varphi_\lambda(r)=\frac{\sqrt2}{2\pi}\int_{-r}^r
 e^{i\lambda s}(\cosh r-\cosh s)^{-1/2}\dd s.
\end{equation}
Suppose $0<r\le1$ and $\lambda r\ge1$. Choose a smooth partition in $s/r$ into two endpoint pieces and one piece supported away from the endpoints. Near $s=r$, write $s=r-v$. The amplitude takes the form
\begin{equation}\label{eq:endpoint-amplitude}
 (\cosh r-\cosh(r-v))^{-1/2}
 =(\sinh r)^{-1/2}v^{-1/2}h_r(v),
\end{equation}
where $h_r$ is smooth up to $v=0$ on $0\le v\le r/2$ and
\[
 |\partial_v^j h_r(v)|\le C_jr^{-j}.
\]
Indeed, $(\cosh r-\cosh(r-v))/(v\sinh r)$ is a positive smooth function of $v/r$ and $r$ on this compact region, bounded above and below uniformly. Include the endpoint cutoff in $h_r$, so that it is supported where $v\le r/2$.

After factoring out $e^{i\lambda r}(\sinh r)^{-1/2}$, the endpoint integral is
\[
 A_+(\lambda,r)=\int_0^{r/2}e^{-i\lambda v}v^{-1/2}h_r(v)\dd v.
\]
For each $j\ge0$, differentiation inserts a factor $(-iv)^j$. On $v\lesssim\lambda^{-1}$, direct integration bounds the result by $C_j\lambda^{-j-1/2}$. On a dyadic interval $v\asymp V=2^k/\lambda$, derivatives of the localized amplitude are bounded by $C_{j,M}V^{j-1/2-M}$, since $V\lesssim r$. Integrating by parts $M$ times therefore bounds that piece by
\[
 C_{j,M}\lambda^{-M}V^{j+1/2-M}
 =C_{j,M}\lambda^{-j-1/2}2^{k(j+1/2-M)}.
\]
The sum converges when $M>j+1/2$. Thus
\begin{equation}\label{eq:smallr-endpoint-symbol}
 |\partial_\lambda^jA_+(\lambda,r)|\le C_j\lambda^{-j-1/2}.
\end{equation}
The other endpoint gives the same estimate with phase $e^{-i\lambda r}$.

For the interior piece $R(\lambda,r)$, the substitution $s=ru$ produces a smooth amplitude supported in a fixed compact subinterval of $(-1,1)$, with all $u$ derivatives bounded uniformly in $r$. Integration by parts gives, for every $j,M$,
\[
 |\partial_\lambda^jR(\lambda,r)|\le C_{j,M}r^j(\lambda r)^{-M}.
\]
Assign this piece to the positive phase by writing $R=e^{i\lambda r}(e^{-i\lambda r}R)$. Derivatives of the parenthesized amplitude satisfy the same bound after changing constants. Since $\lambda r\ge1$, choosing $M\ge j+1$ gives
\[
 |\partial_\lambda^j(e^{-i\lambda r}R)|
 \lesssim_j\lambda^{-j-1/2}r^{-1/2}
 \asymp\lambda^{-j-1/2}(\sinh r)^{-1/2}.
\]
Together with \eqref{eq:smallr-endpoint-symbol}, this proves the required two-phase bounds for $r\le1$.

\subsection{Distances at least one}

The radial eigenfunction equation is
\[
 \varphi_\lambda''+\coth r\,\varphi_\lambda'
 +(\lambda^2+\tfrac14)\varphi_\lambda=0.
\]
For $w=(\sinh r)^{1/2}\varphi_\lambda$ it becomes
\begin{equation}\label{eq:jost-ode}
 w''+\left(\lambda^2+\frac1{4\sinh^2r}\right)w=0,
 \qquad
 \frac1{4\sinh^2r}=\sum_{j\ge1}j e^{-2jr}.
\end{equation}
A solution with positive outgoing phase is
\begin{equation}\label{eq:jost-series}
 w_+(r,\lambda)=e^{i\lambda r}\sum_{k\ge0}a_k(\lambda)e^{-2kr},
 \qquad a_0=1,
\end{equation}
where substitution in \eqref{eq:jost-ode} gives
\begin{equation}\label{eq:jost-rec}
 4k(k-i\lambda)a_k(\lambda)
 =-\sum_{j=1}^k j\,a_{k-j}(\lambda),\qquad k\ge1.
\end{equation}
For real $\lambda\ge1$, consider the complex disk $|z-\lambda|\le\lambda/2$. If $k\le\lambda$, then $|k-iz|\ge\Re z\ge\lambda/2\ge k/2$; if $k\ge\lambda$, then $\Re(k-iz)\ge k-\lambda/2\ge k/2$. Induction in \eqref{eq:jost-rec} consequently gives
\[
 |a_k(z)|\le1,
\]
because $\sum_{j=1}^kj=k(k+1)/2\le2k^2$. Cauchy's estimate yields
\begin{equation}\label{eq:jost-symbol}
 |\partial_\lambda^m a_k(\lambda)|\le m!(2/\lambda)^m.
\end{equation}
The series in \eqref{eq:jost-series}, and all its $\lambda$ derivatives, converge uniformly for $r\ge1$. Its amplitude is therefore a symbol of order zero, uniformly in $r$. Complex conjugation supplies the negative-phase solution.

The connection coefficients for the normalized spherical function are expressed through
\begin{equation}\label{eq:c-function}
 \mathbf c(\lambda)=\frac{\Gamma(i\lambda)}{\sqrt\pi\,\Gamma(1/2+i\lambda)},
 \qquad \mathbf c(-\lambda)=\overline{\mathbf c(\lambda)}.
\end{equation}
In the normalization \eqref{eq:jost-series}, the exact connection formula is
\[
 w=\frac1{\sqrt2}\bigl(\mathbf c(\lambda)w_++\mathbf c(-\lambda)w_-\bigr).
\]
One can check the coefficients from the regular solution of the radial equation,
\[
 \varphi_\lambda(r)={}_2F_1\left(\tfrac12+i\lambda,\tfrac12-i\lambda;
 1;-\sinh^2(r/2)\right).
\]
The hypergeometric connection formula at infinity \cite[Eq.~15.8.2]{DLMF} gives the coefficient of $e^{(i\lambda-1/2)r}$ as
\[
 4^{1/2-i\lambda}\frac{\Gamma(2i\lambda)}{\Gamma(1/2+i\lambda)^2}
 =\frac{\Gamma(i\lambda)}{\sqrt\pi\,\Gamma(1/2+i\lambda)},
\]
where the equality is the gamma duplication identity \cite[Eq.~5.5.5]{DLMF}. Passing from $\varphi_\lambda$ to $w$ introduces the factor $1/\sqrt2$. The gamma-ratio estimate can also be differentiated without a formal
termwise differentiation of an asymptotic series.  For real $\lambda\ge2$,
consider the disk $|z-\lambda|\le\lambda/4$.  The points $iz$ remain in a
fixed closed sector separated from the negative real axis, and their
moduli are comparable to $\lambda$.  The uniform gamma-ratio form of Stirling\textquotesingle s formula \cite[Eq.~5.11.13]{DLMF} gives
\[
 \left|\frac{\Gamma(iz)}{\Gamma(1/2+iz)}\right|
 \le C\lambda^{-1/2}
\]
on this disk.  The ratio is holomorphic there.  Cauchy's estimate on the
disk of radius $\lambda/4$ gives
\[
 |\partial_\lambda^m\mathbf c(\lambda)|
 \le C_m\lambda^{-1/2-m}.
\]
The compact parameter interval $1\le\lambda\le2$ is covered by
smoothness.  Complex conjugation gives the estimate for
$\mathbf c(-\lambda)$.  In addition, the uniformly convergent Jost series
may be differentiated in $r$ any fixed number of times for $r\ge1$,
because $\sum_k k^j e^{-2k}<\infty$.  Thus it solves the radial equation
term by term, and its use in the connection formula is justified.

Together with \eqref{eq:jost-symbol}, this gives the following bound for the $m$th $\lambda$ derivative of either amplitude of $\varphi_\lambda$:
\[
 C_m\lambda^{-1/2-m}(\sinh r)^{-1/2}.
\]
To extend the construction to $\lambda r\ge c>0$, it remains to cover $c\le\lambda r\le2$. Here $r\le2$ because $\lambda\ge1$. Differentiation of the boundary integral gives
\[
 |\partial_\lambda^j\varphi_\lambda(r)|
 \le r^j\varphi_0(r)\le r^j.
\]
On this region, take the negative-phase amplitude to be zero and the positive-phase amplitude to be $e^{-i\lambda r}\varphi_\lambda(r)$. Its $j$th derivative is bounded by $C_jr^j\le C_j\lambda^{-j}$. Since $r\le2$ and $\lambda r\le2$, this is bounded by
\[
 C_j\lambda^{-j-1/2}(\sinh r)^{-1/2}.
\]
Use a smooth cutoff in $\lambda r$ that is one for $\lambda r\le1$ and zero for $\lambda r\ge2$ to join this choice to the oscillatory amplitudes. The overlap lies in the region where both constructions are valid, and every derivative of the cutoff contributes at most a constant multiple of $\lambda^{-1}$.

Finally, $\lambda\tanh(\pi\lambda)$ is a symbol of order one on $[1,\infty)$, so multiplication proves \eqref{eq:b-symbol}. A smooth partition near $r=1$ combines the two constructions. Fixed changes in the lower bound on $\lambda r$ change only the constants.

\end{document}